\documentclass[reqno, 11pt]{amsart}
\usepackage{mathrsfs,graphicx,url, amssymb, pinlabel}
\usepackage[usenames,dvipsnames]{xcolor}
\usepackage[colorlinks=true,linkcolor=Black,citecolor=Black, urlcolor=Black]{hyperref}

\usepackage[letterpaper,margin=1in]{geometry}

\newtheorem{lemma}{Lemma}
\newtheorem*{theorem}{Theorem}

\DeclareMathOperator \esssupp {ess~supp}
\DeclareMathOperator \supp {supp}

\DeclareMathOperator \ellip {ell}

\renewcommand{\Re}{\operatorname{Re}}
\renewcommand{\Im}{\operatorname{Im}}

\newcommand{\norm}[1]{\| #1\|}

\title{Propagation of Singularities for the Wave Equation}
\author{Dean Baskin}
\address{Department of Mathematics, Texas A\&M University, Mailstop
  3368, College Station, TX 77843, USA}
\author{Kiril Datchev}
\address{Mathematics Department, Purdue University, West Lafayette,
  IN 47907, USA}
\begin{document}
\begin{abstract}
This expository note gives a digest version of H\"ormander's propagation of singularities theorem for the wave equation. 
\end{abstract}
\maketitle

\section{Introduction}

Let $P= \Box = - \partial_{x_0}^2 + \partial_{x_1}^2 + \cdots +
\partial_{x_n}^2$ be the flat D'Alembertian on $\mathbb R^{d}$,
$d=n+1$. We use coordinates $x=(x_0,\dots, x_n)$. H\"ormander's
propagation of singularities result for solutions to $\Box u = 0$ says
that if $u$ has a given degree of Sobolev regularity at a given
position and direction in spacetime $(x, \hat \xi) \in \mathbb R^{d}
\times \mathbb S^{d-1}$, then it must have the same Sobolev regularity
at all $(x',\hat \xi)$ such that $x'$ is on the light ray emanating from
$x$ with direction $\hat \xi$.

To make precise the notion of a given degree of Sobolev regularity at a given position and direction, we introduce the following definitions. For $m \in \mathbb R$, let $S^m = S^m(\mathbb R^{d})$ be the \textit{Kohn--Nirenberg} class of symbols, defined by the condition that $a \in S^m$ means $a \in C^\infty (\mathbb R^{d} \times \mathbb R^{d})$ and the partial derivatives of $a$ obey the bounds
\[
| \partial_x^\alpha \partial_\xi^\beta a(x,\xi)| \le C_{\alpha,\beta} \langle \xi \rangle^{m-|\beta|}.
\]

To each such symbol we associate a \textit{pseudodifferential operator} $A$ given by
\begin{equation}\label{e:pseudodef}
 A u (x) = \frac 1 {(2\pi)^{d}} \int \!\!\int e^{i(x-y)\cdot \xi} a(x,\xi) u(y) \, dy\, d\xi.
\end{equation}
The set of $A$ corresponding to some $a \in S^m$ is denoted $\Psi^m$. If $A \in \Psi^m$, then $A$ is a bounded operator from $H^k$ to $H^{k-m}$ for every $k$. For a proof of this general statement see Corollary 4.32 of \cite{hintz}, but note the following simpler cases: if $a(x,\xi) = a(x)$ is independent of $\xi$ then $A$ is the multiplier $u(x) \mapsto a(x)u(x)$, and if $a(x,\xi) = \langle \xi \rangle ^{m}$ then $\|Au\|_{L^2}$ is the usual norm on the Sobolev space $H^m$ given in terms of the Fourier transform.

To each $a \in S^m$  which is compactly supported in $x$, we associate
an \textit{essential support at fiber infinity}\footnote{One can extend this definition to symbols which need not be compactly supported in $x$: see Chapter 6 of \cite{hintz} or Section E.2.1 of \cite{dyatlovzworski}.}
(i.e. a support as the frequency $\xi$ tends to infinity), denoted by $\esssupp(a)$, and defined as follows: $(x,\hat \xi)$ is not in the essential support if and only if there is a neighborhood $U \subset \mathbb R^d \times \mathbb S^{d-1}$ of $(x,\hat \xi)$, such that the  partial derivatives  of $a$ obey the bounds
\[
| \partial_x^\alpha \partial_\xi^\beta a(x',\xi')| \le  C_{\alpha,\beta,N} \langle \xi' \rangle^{-N}
\]
for all $N$ when $(x',\xi'/|\xi'|) \in U$.

We say $a \in S^m$ is \textit{elliptic} at $(x,\hat \xi)$ if there  are positive constants $C$ and $\varepsilon$,  and a neighborhood $U \subset \mathbb R^d \times \mathbb S^{d-1}$ of $(x,\hat \xi)$, such that 
\[
 |a(x',\xi')| \ge \varepsilon \langle \xi' \rangle^m,
\]
when $(x',\xi'/|\xi'|) \in U$ and $|\xi| \ge C$.  The set of points in $\mathbb R^d \times \mathbb S^{d-1}$ at which $a$ is elliptic is denoted $\ellip(a)$.

Given symbols $b$, $e$, and $g$, we say that $b$ is \textit{controlled
  by} $e$ \textit{through} $g$ if for each point $(x,\hat \xi)$ in
$\esssupp(b)$, there is a light ray contained in $\ellip(g)$ which
starts at some point $(x',\hat \xi')$ in $\ellip(e)$ and ends at
$(x,\hat \xi)$. A \textit{light ray} is an integral curve of the
\textit{Hamilton vector field} $H_P$ lying in the \emph{characteristic
  set}
$\Sigma (P) = \{ (x,\hat{\xi})\in \mathbb{R}^{d}\times
\mathbb{S}^{d-1} : \hat{\xi}_{0}^{2}-\sum \hat{\xi}_{j}^{2} = 0\}$.
Here the Hamilton vector field is given by
\[
 H_P = 2 \xi_0 \partial_{x_0} -2 \xi_1 \partial_{x_1} - \cdots - 2 \xi_n \partial_{x_n}.
\]
Note that the Hamilton vector field has no $\xi$ derivatives, so if
$(x, \hat{\xi})$ is connected by a light ray to $(x', \hat{\xi}')$,
then $\hat{\xi} = \hat{\xi}'$.  See  Figure~\ref{fig:propsingpic}.

\begin{figure}[ht]
  \label{fig:propsingpic}
\labellist
\small
\pinlabel $\ellip(g)$ [l] at 115 21
\pinlabel $\ellip(e)$ [l] at 47 34
\pinlabel $\esssupp(b)$ [l] at 168 48
\endlabellist
 \includegraphics[width=10cm]{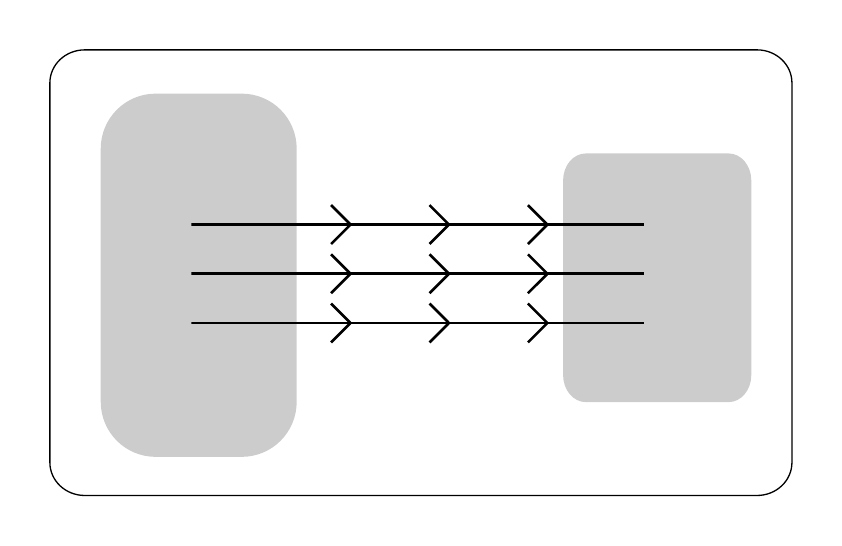}
\caption{
The horizontal coordinate $s$ is chosen so that the Hamilton vector field is $\partial_s$, thus the arrows are integral curves of the Hamilton vector field.}
\end{figure}

We are now ready to state  H\"ormander's propagation of singularities theorem.
\begin{theorem}
  \label{main-thm}
  Let $b \in S^k$, $e \in S^k$, $g \in S^{k-1}$ for some real
  $k$ be compactly supported in $x$. Suppose $u\in H^{-N}$, $Eu \in L^{2}$, and $GPu \in L^{2}$.  If
  $b$ is controlled by $e$ through $g$, then $Bu \in L^{2}$ and
  there is a constant $C$ (independent of $u$) such that
\begin{equation}\label{e:propsing} 
\|B u\|_{L^2} \le C( \|Eu\|_{L^2} + \|G Pu\|_{L^2} + \|u\|_{H^{-N}}). 
\end{equation}
\end{theorem}


Since $B$ and $E$ both map
$H^k \to L^2$, this says in particular that if we have a solution to
$\Box u=0$ which is in some Sobolev space $H^{-N}$, and if we happen
to know that it is in some better Sobolev space $H^k$ in the positions
and directions given by the elliptic set of $e$, then it is also in
this better Sobolev space $H^k$ at points and
directions 
emanating from that set. In this sense \eqref{e:propsing} is a
propagation of regularity estimate. Since a singularity is a place
where regularity is lacking, we also call \eqref{e:propsing} a
propagation of singularities estimate.

As $b$ is compactly supported in $x$, without loss of generality
we can take $e$ and $g$ also to have compact support in $x$ and then
$u$ need only lie in $H^{-N}$ locally.  In particular, if
$\chi\in C^{\infty}_{c}$ satisfies $\chi \equiv 1$ on the support of
$g$, then
\begin{equation*}
  \norm{B_{0}u}\leq C \left( \norm{Eu} + \norm{GPu} + \norm{\chi u}_{H^{-N}}\right).
\end{equation*}


For example, let $u = \delta(x_0 - x_1)$. Then $u \in H^s_{\text{loc}}$ if and only if $s<1/2$. Microlocally, however, one can show that, for almost all $(x,\hat \xi)$, we have $u \in H^s$ at $(x,\hat \xi)$ for all $s$. More precisely, $u$ is in these better Sobolev spaces at $(x,\hat \xi)$ if and only if  $x_0 \ne x_1$ or $|\hat \xi_0| \ne 1/ \sqrt 2$ or $\hat \xi_1 \ne - \hat \xi_0$.

This theorem was first proved by using  a Fourier integral operator to microlocally conjugate $P$ to the simpler operator $\partial_{x_0}/i$, for which the theorem is easy: see \cite{hormanderfio} and \cite{dh}, specifically Section 6 of the latter paper. The proof we present here uses a positive commutator argument with a
pseudodifferential commutant. It is a special case of H\"ormander's
proof in Section 3.5 of \cite{hormander}. Other presentations of this
proof in more general settings than the one here can be found in
Section E.4 of \cite{dyatlovzworski} and Section 8 of
\cite{hintz}. 

The crux of the proof involves bounding the pairing
\begin{equation}\label{e:aapapa}
    \Im \langle Au, APu\rangle = \frac{1}{2i} \left( \langle PA^{*}Au,
      u\rangle - \langle A^{*}APu, u\rangle\right) = \left\langle \frac{1}{2i}[P,A^{*}Au]u,u\right\rangle.
\end{equation}
for a well-chosen $A\in \Psi^{k-\frac{1}{2}}$. We bound \eqref{e:aapapa} from below via a
straightforward application of Cauchy--Schwarz, and from above by arranging that
the commutator $[P, A^{*}A]/i$ is almost negative. 

As we will see below in Section \ref{s:micro}, the principal symbol of $[P,A^*A]/i$ is $H_P(|a|^2)$: this is the fundamental relationship between commutators and Hamilton vector fields. We will see that to arrange that $[P, A^{*}A]/i$ is almost negative in the right sense we must arrange that $H_P(|a|^2)$ is negative on $\esssupp(b) \setminus \ellip(e)$, i.e. that $|a|^2$ is decreasing along the Hamilton flow in this region. Such an $a$ is called an \textit{escape function}, a terminology which goes back to \cite{mrs}. Constructing an escape function is straightforward here because of the simple geometry of our spacetime.

Theorem~\ref{main-thm} has many generalizations.  Most directly, a
nearly identical proof strategy applies to operators of real principal
type that are not necessarily self-adjoint: i.e. only the principal symbol of $P$ needs to be real, and $P$ can be a variable coefficient differential operator or even a pseudodifferential operator. Constructing an escape function is then done by locally straightening out the Hamilton flow. The additional error term arising in equation~\eqref{e:aapapa}
from $P-P^{*}$ is a lower order term.  One can also treat cases where the principal
part of $P$ is not real, provided the imaginary part has a sign; this
sign dictates the direction in which we can propagate regularity.

A common use for such propagation results is to help obtain global
estimates for solutions of wave-like equations.  As an example, in
the non-trapping setting, these estimates let you propagate regularity
from where it is a priori known (e.g., in the distant past for the forward
propagator).  Closing the estimate, however, typically requires
additional estimates, which can be more complicated in settings with trapping.


\section{Preliminaries from microlocal analysis}\label{s:micro}

The important calculations take place on the level of symbols. To translate the results into statements about operators, we use the following formula for the symbol of the composition of two pseudodifferential operators. If $A_1 \in \Psi^{m_1}$ and $A_2 \in \Psi^{m_2}$ have symbols $a_1$ and $a_2$ respectively, then the composition $A_1A_2$ belongs to $\Psi^{m_1+m_2}$ and has symbol given by $a_3 \in S^{m_1+m_2}$ such that
\begin{equation}\label{e:composition}
r:=a_3 - a_1a_2 - \frac 1 i \partial_\xi a_1 \cdot \partial_x a_2\in S^{m_1+m_2 -2},
\end{equation}
and $\esssupp (r) \subset \esssupp (a_1) \cap \esssupp(a_2)$.

We define the \emph{principal symbol} $\sigma_{m}(A)$ of an operator
$A\in \Psi^{m}$ to be the equivalence class of its symbol $a$ in
$S^{m}/S^{m-1}$.  An immediate consequence of the composition formula
above is the observation that
\begin{equation*}
  \sigma_{m_{1}+m_{2}}(A_{1}A_{2}) = \sigma_{m_{1}}(A_{1})\sigma_{m_{2}}(A_{2}).
\end{equation*}
Similarly, for $A_{1}\in \Psi^{m_{1}}$ and $A_{2}\in \Psi^{m_{2}}$,
the commutator $[A_{1},A_{2}] = A_{1}A_{2}-A_{2}A_{1}$ lies in
$\Psi^{m_{1}+m_{2}-1}$ and satisfies
\begin{equation*}
  \sigma_{m_{1}+m_{2}-1}\left( [A_{1},A_{2}]\right) = \frac{1}{i} \left(
    \frac{\partial \sigma_{m_{1}}(A_{1})}{\partial\xi}
    \frac{\partial \sigma_{m_{2}}(A_{2})}{\partial x} -
    \frac{\partial \sigma_{m_{2}}(A_{2})}{\partial\xi}
    \frac{\partial \sigma_{m_{1}}(A_{1})}{\partial x}\right) = \frac 1 i H_{\sigma_{m_{1}}(A_{1})}\left(\sigma_{m_{2}}(A_{2})\right), 
\end{equation*}
where the Hamilton vector field $H_{\sigma_{m_{1}}(A_{1})}$ is defined by the equation.

Concretely, the principal symbol of $P = \Box$ is
\begin{equation*}
  \sigma_{2}(P) = \xi_{0}^{2} - \xi_{1}^{2} - \dots - \xi_{n}^{2},
\end{equation*}
and its Hamilton vector field is
\begin{equation*}
  H_{P} = 2\xi_{0}\partial_{x_{0}} - 2\xi_{1}\partial_{x_{1}} - \dots -
  2\xi_{n} \partial_{x_{n}}.
\end{equation*}

We also use the following adjoint formula: If $A \in \Psi^m$ has symbol $a$, then the adjoint operator has symbol $a' \in S^m$ obeying
\[
 a' - \overline{a} \in S^{m-1}.
\]
The composition formula can be found in Theorem 4.16 of \cite{hintz} and the adjoint formula can be found in Corollary 4.13 of \cite{hintz}. They are easy to check for differential operators, i.e. in the case that the symbols are polynomials in $\xi$. In general they may be checked by writing out the definitions and Taylor expanding.\footnote{Observe that if $A_3=A_1A_2$ then by definition $A_3 u(x)$ is given by \eqref{e:pseudodef} with $a(x,\xi)$ replaced by 
\[
  a_3(x,\xi) = \frac 1 {(2\pi)^{d}}  \int\!\! \int e^{-ix\cdot \xi}e^{i x\cdot \eta} e^{-iz\cdot \eta} e^{iz \cdot\xi} a_1(x,\eta) a_2(z,\xi)\,d\eta \,dz = \frac 1 {(2\pi)^{d}}  \int\!\! \int e^{-iw \cdot \zeta} a_1(x,\xi + \zeta) a_2(x+w,\xi)\,d\zeta \,dw,
\]
where we used $ e^{-ix\cdot \xi}e^{i x \cdot \eta} e^{-iz\cdot \eta} e^{iz \cdot \xi} = e^{i(x-z)\cdot (\eta - \xi)}$, substituted $ w = z-  x$ and $\zeta = \eta - \xi$, and ignored issues of convergence. To deduce the expansion \eqref{e:composition}, Taylor expand
\[
a_1(x,\xi + \zeta) = a_1(x,\xi) + \zeta \cdot \partial_\xi a_1(x,\xi) + \cdots, \qquad a_2(x+w,\xi) = a_2(x,\xi) + w \cdot \partial_{x}a_2(x,\xi) + \cdots,
\]
and use the fact that $\frac 1 {(2\pi)^d} \int \int e^{-iw\zeta} w^\alpha \zeta^\beta \,d\zeta \, d w = i^{|\beta|}\int w^\alpha \partial^\beta \delta(w)  dw$ which is $(-i)^{|\beta|}$ if $\alpha = \beta$ and $0$ otherwise. The convergence issues can be handled by a partition of unity, and the adjoint formula can be proved in the same way: see Chapters 8 and 9 of \cite{wong} for more.
}

Our first application of the composition formula is to the following elliptic estimate.

\begin{lemma}
Let $a \in S^m$ and $a' \in S^{m'}$ be such that $a$ is compactly
supported in $x$ and $\esssupp a \subset \ellip(a')$.  For any $k$
and $N$, if $u \in H^{-N}$ and $A'u \in H^{k+m-m'}$, then $Au \in
H^{k}$ and there is $C$ (independent of $u$) such that
\begin{equation}\label{e:elliptic}
 \|A u\|_{H^k} \le C (\|A'u\|_{H^{k+m-m'}} + \|u\|_{H^{-N}}).
\end{equation}
\end{lemma}

\begin{proof}
Because $a'$ is elliptic on the essential support of $a$, there exists $g \in S^{m-m'}$ such that $ga' = a$ when $\xi$ is large.\footnote{To construct such a $g$, for $C$ as in the definition of ellipticity, let
$g(x,\xi) a'(x,\xi)= (1- \chi_C(\xi))\psi(x,\hat \xi)a(x,\xi), $
where $\chi_C \in C_c^\infty(\mathbb R^d)$ is identically $1$ on the ball of radius $C$ centered at $0$, $\psi \in C^\infty(\mathbb R^d \times \mathbb S^{d-1})$, $\supp \psi \subset \ellip(a')$, and $\psi = 1$ near $\esssupp(a)$.} By the composition formula \eqref{e:composition}, there is $R_1 \in \Psi^{m-1}$ such that
\[
 G A' = A + R_1.
\]
Since $G\colon  H^{k+m - m'}\to H^k$ is bounded, it follows that
\[
 \|Au\|_{H^k} \le C (\|A'u\|_{H^{k+m-m'}} + \|R_1u\|_{H^k}).
\]
This gives \eqref{e:elliptic} when $-N = k+m-1$. To get it for larger values of $N$ we apply the same reasoning with $A$ replaced by $R_1$. That gives $R_2 \in \Psi^{m-2}$ such that
\[
 \|R_1 u\|_{H^k} \le C(\|A' u\|_{H^{k+m-1-m'}} + \|R_2 u\|_{H^{k}}),
\]
which implies \eqref{e:elliptic} when $-N = k+m-2$. Repeating this argument gives \eqref{e:elliptic} for arbitrary $N$.
\end{proof}

We further require G{\aa}rding's inequality, which states that
operators with non-negative principal symbols are non-negative to
leading order.  In other words, if $A\in \Psi^{m}$ is compactly
supported in $x$ and has $\sigma_{m}(A) \geq 0$ for $|\xi|$
sufficiently large, then there is some constant $C$ so that
\begin{equation*}
  \Re \langle Au, u\rangle \geq - C\| u\|^{2}_{H^{\frac{m-1}{2}}}.
\end{equation*}
The proof of this general form of G{\aa}rding's inequality is somewhat
involved (see Theorem 18.1.14 of \cite{hvol3}), but is straightforward when the principal symbol of $A$ is a square (or sum of
squares).  Indeed, if $\sigma_{m}(A) = b^{2}$, then by the composition
and adjoint formulas, we may write $A = B^{*}B + R$, where $B\in
\Psi^{\frac{m}{2}}$ and $R \in \Psi^{m-1}$.  It then follows that
\begin{equation*}
  \Re \langle Au, u\rangle = \| Bu\|^{2} + \langle Ru , u\rangle.
\end{equation*}
The first term on the right side is non-negative.  Writing $R$ as the
composition of two operators of order $\frac{m-1}{2}$ (such as
$R = (I - \Delta)^{\frac{m-1}{4}} \circ (I -
\Delta)^{-\frac{m-1}{4}}R$) then shows that this last term is bounded
below by $-C \| u\|_{H^{\frac{m-1}{2}}}^{2}$ for some $C$.  Observe further that
combining this proof with the elliptic estimate \eqref{e:elliptic} shows that if
$A\in \Psi^{m}$ is compactly supported in $x$, has non-negative
principal symbol, and $B\in \Psi^{0}$ has $\esssupp a \subset
\ellip(b)$, then for any $N$ there is a constant $C$ with
\begin{equation}
  \label{e:garding-micro}
  \Re \langle A u, u\rangle \geq - C \| Bu\|_{H^{\frac{m-1}{2}}}^{2} - C
  \| u\| _{H^{-N}}^{2}.
\end{equation}

\section{Proof of theorem}

In this section we use the notation $(z,\zeta)$ for a variable point of $\mathbb R^d \times \mathbb R^d$, to avoid a typographical collision with $(x,\hat \xi)$ and $(x',\hat \xi')$ which we use to denote points in the essential support or elliptic set of a particular symbol.

The main lemma in the proof of the theorem is the following similar but weaker statement.

\begin{lemma}
  \label{main-lemma}
Let $b \in S^k$, $e \in S^k$, $g \in S^{k-1}$ for some real $k$ be
compactly supported in $x$.  Let $\Lambda_{s} =
  (I - \Delta_{\mathbb{R}^{d}})^{s/2} \in \Psi^{s}$. If $b$ is controlled by $e$ through $g$, then for every $N$ there is a constant $C$ such that
   \begin{equation}\label{e:bug1}
    \norm{Bu}^{2} \leq C \left( \norm{\Lambda_{1-k}Gu}_{H^{k-\frac{1}{2}}}^{2}
      + \norm{Eu}^{2} + \norm{GPu}^{2} + \norm{u}_{H^{-N}}^{2}\right),
  \end{equation}
for all functions $u \in H^{k+1}$.
\end{lemma}

\begin{proof}
We assume without loss of generality that $\esssupp (e) \subset \ellip (g)$; this
can be arranged to hold without changing the hypotheses by shrinking
the essential support of $e$.  An application of the elliptic
estimate~\eqref{e:elliptic} then yields the result with the original
$E$. 

Similarly, by a partition of unity argument, it is enough to prove that, for each point $(x, \hat \xi)$ in $\esssupp(b)$, there is $b_1 \in S^k$ which is elliptic at $(x, \hat \xi)$ such that \eqref{e:bug1} holds with $B$ replaced by $B_1$.

For $A \in \Psi^{k - \frac 12}$ to be specified later, consider the
  pairing
  \begin{equation*}
    \Im \langle Au, APu\rangle = \frac{1}{2i} \left( \langle PA^{*}Au,
      u\rangle - \langle A^{*}APu, u\rangle\right) = \left\langle \frac{1}{2i}[P,A^{*}Au]u,u\right\rangle.
  \end{equation*}

1. We first bound the pairing below.
By Cauchy--Schwarz, for any
  $\epsilon > 0$ we have
  \begin{equation*}
       \Im \langle Au, APu\rangle  \ge  -\frac{1}{4\epsilon} \norm{APu}_{H^{-1/2}}^{2} - \epsilon \norm{Au}_{H^{1/2}}^{2}.
  \end{equation*}
  To write this bound in terms of the $L^{2}$ pairing we use $\|\Lambda_{s} u\|_{L^2} = \|u\|_{H^s}$. That gives
  \begin{equation}\label{e:pairingbelow}
    \left\langle \frac{1}{2i}[P,A^{*}Au]u,u\right\rangle \ge -\frac{1}{4\epsilon}\norm{\Lambda_{-1/2} A Pu}_{L^{2}}^{2} - \epsilon \norm{\Lambda_{1/2}Au}_{L^{2}}^{2}.
  \end{equation}

2. We next bound the pairing above. Recall that
$\frac{1}{2i}[P,A^{*}A] \in \Psi^{2k}$ and has principal symbol
given by $\frac 12 H_{P}(a^{2}) = aH_p(a)$. We will construct $a$ in such a way that there exist a compact set $K \subset \ellip(e)$ and a $\gamma >0$ such that
\begin{equation}\label{e:ahp}
- aH_p(a)  - \gamma \left( \langle \zeta
      \rangle^{1/2}a\right)^{2} \ge 0, \qquad \text{ off of }K.
\end{equation}
Hence there exists $C$ large enough that\footnote{In our setting,
  $\gamma>0$  can be arbitrary, but when dealing with error terms
  $\gamma$ must be taken sufficiently large.}
  \begin{equation*}
    C e^{2} - aH_p(a) - \gamma \left( \langle \zeta
      \rangle^{1/2}a\right)^{2} \ge 0.
  \end{equation*}
In our construction of $a$ we will also obtain
\begin{equation}\label{e:aesssupp}
 \esssupp(a) \subset \ellip(g).
\end{equation}
Using the version of the G{\aa}rding inequality
  stated in~\eqref{e:garding-micro}\footnote{If we want to avoid invoking the full strength of the sharp G\aa rding inequality as in Theorem 18.1.14 of \cite{hvol3}, we must check that the quantity $C e^{2} - aH_p(a) - \gamma \langle \zeta
      \rangle a^{2}$ is a sum of squares. For that, take $\psi \in S^0$ such that $\psi = 0$ near $K$ and $\psi = 1$ near the complement of $\ellip(e)$, and use \eqref{e:princip} to write $C e^{2} - aH_p(a) - \gamma \langle \zeta
      \rangle a^{2} = a_1^2 + a_2^2$, where
\[
 a_1 = e\sqrt{C + (1-\psi^2)( - aH_p(a) - \gamma \langle \zeta
      \rangle a^{2})/e^{2}}, \qquad a_2 = \psi \langle \zeta \rangle^k \chi \sqrt{(-2\varphi'-\gamma\varphi)  \varphi}.
\]
To see that $a_1$ is smooth use the fact that the square root of a positive function is smooth. To see that $a_2$ is smooth, use the explicit formula for $\varphi$.
}
 then yields
  \begin{equation}\label{e:pairingabove}
  \left\langle\frac{1}{2i}[P,A^{*}A]u,u\right\rangle  \leq C
    \norm{Eu}^{2} + C\norm{\Lambda_{1-k}Gu}_{H^{k-\frac{1}{2}}}^{2} + C
    \norm{u}_{H^{-N}} - \gamma \norm{\Lambda_{1/2}Au}^{2}.
  \end{equation}

3. We now combine the two bounds on the pairing.  Putting together \eqref{e:pairingbelow} and \eqref{e:pairingabove} and taking $\epsilon = \gamma/2$
yields
  \begin{equation*}
    \norm{\Lambda_{1/2}Au}^{2} \leq C\left(
      \norm{\Lambda_{1-k}Gu}_{H^{k-\frac{1}{2}}}^{2} + \norm{Eu}^{2} +
      \norm{\Lambda_{-1/2}APu}^{2} + \norm{u}_{H^{-N}}^{2}\right).
  \end{equation*}
To deduce \eqref{e:bug1} we use the fact that the elliptic
  estimate~\eqref{e:elliptic} and \eqref{e:aesssupp} imply  $\norm{\Lambda_{-1/2}APu}^{2} \le C(\norm{GPu}^{2} + \|u\|^2_{H^{-N}})$, and moreover we can use $\Lambda_{1/2}A$ as $B_1$ provided we construct $a$ so that
\begin{equation}\label{e:aellip}
 (x,\hat \xi) \in \ellip(a).
\end{equation}
Thus is remains to construct $a \in S^{k-\frac 12}$ such that \eqref{e:ahp}, \eqref{e:aesssupp}, and \eqref{e:aellip} all hold.

4. To construct $a$, use the fact that
  there is a light ray contained in $\ellip (g)$ starting at some
  $(x',\hat{\xi})\in \ellip(e)$ ending at $(x,\hat{\xi})$, and this light ray can be
  parametrized by
\begin{equation*}
  t \mapsto \left( x_{0} \pm t \hat{\xi}_{0}, x_{j} \mp t
    \hat{\xi}_{j}, \hat{\xi}\right),
\end{equation*}
with the choice of sign depending on whether the flow from
$(x', \hat{\xi})$ to $(x,\hat{\xi})$ is in the direction the flow of
the Hamilton vector field or its opposite.  We assume for the argument
below that it is with the direction of the Hamilton vector field (and
hence the top sign is chosen).

Recall that $(x,\hat{\xi})$ lies on a light ray, and because light rays must lie in the characteristic set of
$P$, $\zeta_{0} \neq 0$ on any light ray. Using also the fact that $\ellip(g)$ and $\ellip(e)$ are open sets, fix $t_{0}>0$,
$\delta > 0$, and an open neighborhood $U \subset \mathbb R ^{d-1} \times
\mathbb S^{d-1}$ of $(x,\hat{\xi})$ such that $\hat{\zeta}_{0}$
is non-vanishing on $U$, such that
\begin{equation}\label{e:Uellg}
 t \in (-t_{0}-\delta, \delta) \text{ and } (y_{1},\dots,
  y_{n}, \hat{\zeta})\in U \qquad \Longrightarrow \qquad (x_{0} + t \hat{\zeta}_{0}, y_{j} - t \hat{\zeta}_{j},
    \hat{\zeta}) \in \ellip (g), 
\end{equation}
and such that
\begin{equation}\label{e:Uelle}
 t \in (-t_{0}-\delta, t_0+\delta) \text{ and } (y_{1},\dots,
  y_{n}, \hat{\zeta})\in U \qquad \Longrightarrow \qquad      (x_{0} + t \hat{\zeta}_{0}, y_{j} - t \hat{\zeta}_{j},
    \hat{\zeta}) \in \ellip (e).
\end{equation}

We fix a real function $\chi_1 \in C^{\infty}_{c}(U)$ that is identically $1$
on a neighborhood of $(x,\hat{\xi})$, and put $\chi = \chi_1^2$. Let
\[
 \varphi(t) =  \exp(-\gamma t + (t-\delta)^{-1} - (t+t_0+\delta)^{-1}), \qquad \text{when } -t_0-\delta < t< \delta,
\]
and $\varphi(t) = 0$ otherwise.

\begin{figure}[ht]
  \label{fig:escape-function}
\labellist
\tiny
\pinlabel $-t_0-\delta$ [l] at 100 110
\pinlabel $-t_0$ [l] at 410 110
\pinlabel $0$ [l] at 755 110
\pinlabel $\delta$ [l] at 1030 110
\endlabellist
\includegraphics[width=10cm]{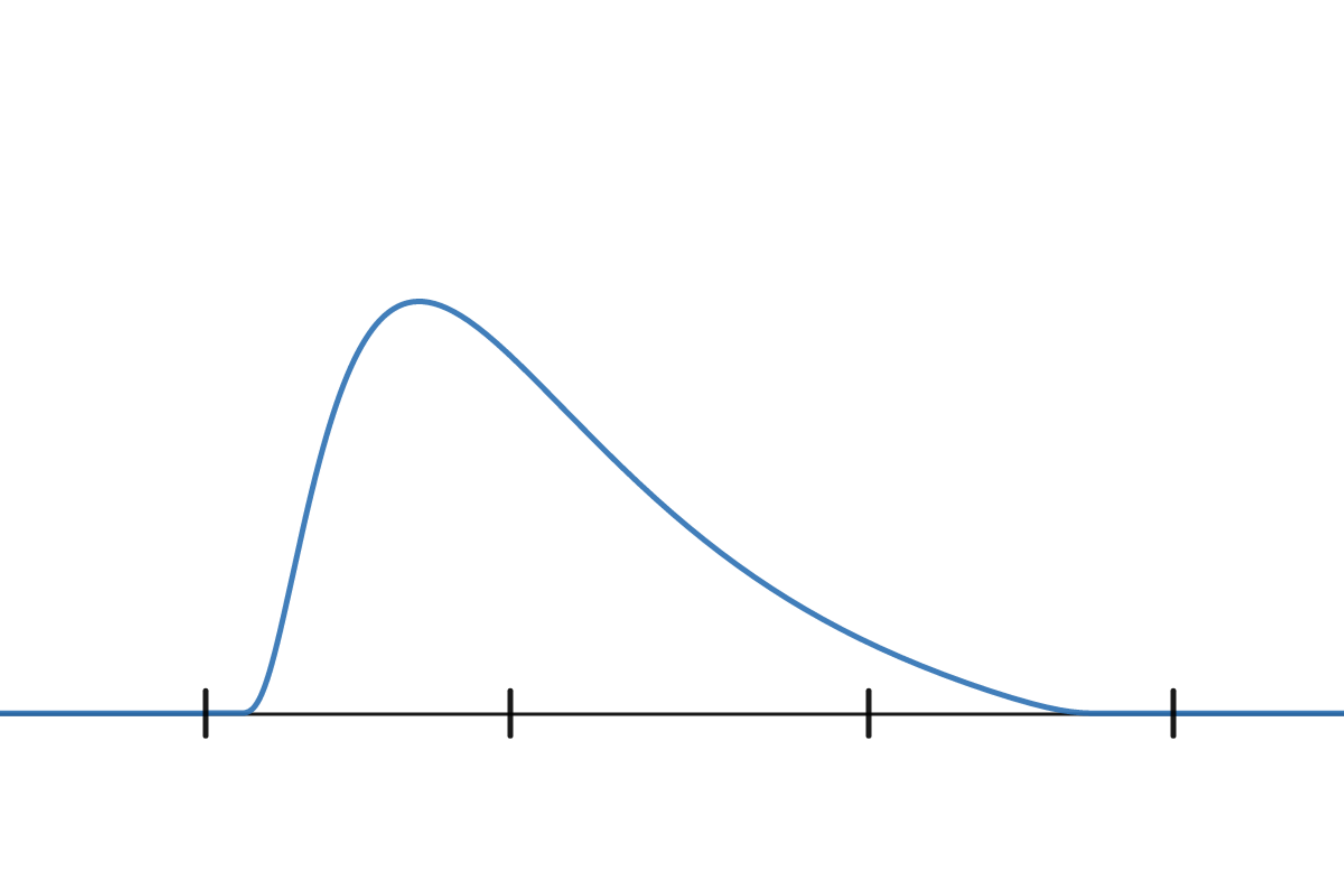} 
\caption{A graph of $\varphi$,  from \url{https://www.desmos.com/calculator/tqmurytqe9}.}
\end{figure}
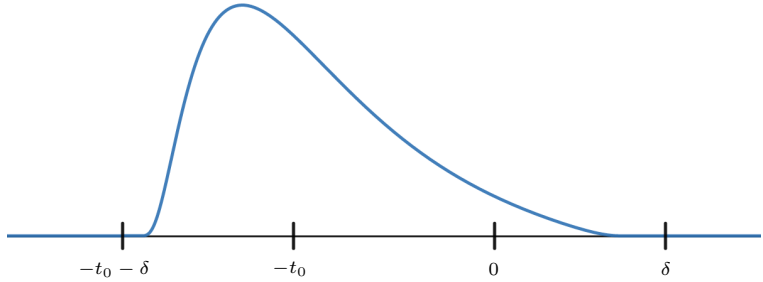

Let
\begin{equation*}
  a(z, \zeta) = \varphi (t) \chi \left( z_{1} + \hat \zeta_1 t \,,\, \dots , z_{n} + \hat \zeta_n t \, , \,
    \hat{\zeta}\right) \langle \zeta\rangle^{k-\frac{1}{2}},
\end{equation*}
where $t = t(z,\zeta) = (z_0 - x_0)/\hat \zeta_0$.
The expression for $a$ is well-defined as $\hat{\zeta}_{0}\neq 0$ on
the support of $\chi$ and it is straightforward to check that $a \in
S^{k-\frac{1}{2}}$. We write more shortly $a = \varphi \chi \langle \zeta\rangle^{k-\frac{1}{2}}$.

Now \eqref{e:aesssupp} follows from the property \eqref{e:Uellg} of $U$, and \eqref{e:aellip} follows from the fact that $a\langle \zeta\rangle^{-k+\frac 12} \ne 0$ at $(x,\xi)$. It remains to check \eqref{e:ahp}, and for this we compute $aH_P(a)$.

The computation is made straightforward by the fact that we have defined $a$ to be the product of $\varphi(t)$ and a function which is constant along flows of $H_P$:
\begin{equation*}
 a H_{P} (a) =  2 \frac {\zeta_0}{\hat \zeta_0} \varphi' \varphi \chi^2 \langle \zeta \rangle^{2k-1},
\end{equation*}
which has principal symbol
\[
  2 \langle \zeta\rangle^{2k} \varphi ' \varphi \chi^{2},
\]
because $\zeta_0/\hat \zeta_0 = |\zeta|$. Thus the principal symbol of $-a H_p(a) - \gamma \langle \zeta \rangle a^2$ is
\begin{equation}\label{e:princip}
 - (2 \varphi' + \gamma \varphi) \varphi \chi^2 \langle \zeta \rangle^{2k},
\end{equation}
and the set where this is negative is contained in the set where  $ - t_0 - \delta \le t \le -t_0$ and in the support of $\chi$, i.e. by the property \eqref{e:Uelle} of $U$ above it is contained in a compact subset of $\ellip(e)$.
\end{proof}

With Lemma~\ref{main-lemma} in hand, we now turn our attention to the
proof of Theorem~\ref{main-thm}.  We proceed by removing the first
term of the right side of the estimate~\eqref{e:bug1} with an
inductive argument and then finally relax the regularity hypothesis
with a regularization argument.

We first consider the inductive argument.  Indeed, we claim that if $b$ is controlled by $e$ through $g$, then
for any $m$, there is a constant so that
\begin{equation}
  \label{eq:inductive}
  \norm{Bu}^{2}\leq C \left(
    \norm{\Lambda_{1-k}Gu}_{H^{k-\frac{m}{2}}}^{2} + \norm{Eu}^{2} +
    \norm{GPu}^{2} + \norm{u}_{H^{-N}}^{2}\right),
\end{equation}
for all $u \in H^{k+1}$. Lemma~\ref{main-lemma} provides the base case $m=1$.  

Suppose now that the inequality~\eqref{eq:inductive} holds for $m$.  As $\ellip(g)$ and
$\ellip (e)$ are open and $\esssupp (b)$ is closed, we may shrink the
support of $g$ to find a new symbol $g_{1} \in S^{k-1}$ so that
\begin{enumerate}
\item $b$ is controlled by $e$ through $g_{1}$, and 
\item $\langle \zeta\rangle ^{1-\frac{m}{2}}g_{1}$ is controlled by
  $\langle \zeta \rangle^{-\frac{m}{2}}e$ through $\langle \zeta\rangle^{-\frac{m}{2}}g$.
\end{enumerate}
The inductive hypothesis yields the estimate
\begin{equation*}
  \norm{Bu}^{2} \leq  C\left(
    \norm{\Lambda_{1-k}G_{1}u}_{H^{k-\frac{m}{2}}}^{2} + \norm{Eu}^{2}+\norm{G_{1}Pu}^{2}+\norm{u}_{H^{-N}}^{2}\right).
\end{equation*}
We then apply Lemma~\ref{main-lemma} to $\langle \zeta
\rangle^{1-\frac{m}{2}}g_{1}\in S^{k-\frac{m}{2}}$, which is
controlled by $\langle
\zeta\rangle^{-\frac{m}{2}}e\in S^{k-\frac{m}{2}}$ through
$\langle\zeta \rangle^{-\frac{m}{2}} g\in S^{k-1-\frac{m}{2}}$, yielding
\begin{equation*}
    \norm{\Lambda_{1-\frac{m}{2}}G_{1}u}^{2} \leq C\left(
                                       \norm{\Lambda_{1-k}Gu}_{H^{k-\frac{m+1}{2}}}^{2}
                                       +
                                       \norm{\Lambda_{-\frac{m}{2}}Eu}^{2}
                                       +
                                       \norm{\Lambda_{-\frac{m}{2}}GPu}^{2}+\norm{u}_{H^{-N}}^{2}
                                       \right).
\end{equation*}
We note first that for any $v \in H^{r}$, $\norm{\Lambda_{r}v} =
\norm{v}_{H^{r}}$.  One consequence is the estimate
\begin{equation*}
  \norm{\Lambda_{-\frac{1}{2}}v} = \norm{v}_{H^{-\frac{1}{2}}} \leq \norm{v}_{L^{2}}.
\end{equation*}
Similarly,
\begin{equation*}
  \norm{\Lambda_{1-k}G_{1}u}_{H^{k-\frac{1}{2}}} = \norm{\Lambda_{\frac{1}{2}}G_{1}u},
\end{equation*}
so that combining the two estimates yields
\begin{equation*}
  \norm{Bu}^{2} \leq C \left( \norm{\Lambda_{1-k}Gu}_{H^{k-\frac{m+1}{2}}}^{2}
  + \norm{Eu}^{2} + \norm{GPu}^{2} + \norm{u}_{H^{-N}}^{2}\right),
\end{equation*}
finishing the inductive step.

For $m \geq 2N+2k$,
\begin{equation*}
  \norm{\Lambda_{1-k}Gu}_{H^{k-\frac{m}{2}}} \leq C\norm{u}_{H^{-N}},
\end{equation*}
allowing us to remove this term. Thus we have \eqref{e:propsing} for all $u \in H^{k+1}$.

We conclude with the regularization argument which proves \eqref{e:propsing} under the weaker hypotheses of the Theorem. We roughly follow  Exercises E.10 and E.31 from
\cite{dyatlovzworski}.  
We  introduce, for fixed $r>0$, a family of regularizing operators
depending on a parameter $\epsilon \in (0,1)$:
\begin{equation*}
  \Lambda_{\epsilon, -r} = (1 - \epsilon^{2}\Delta_{\mathbb{R}^{d}})^{-r/2}.
\end{equation*}
The inverse of $\Lambda_{\epsilon,-r}$ is given by
\begin{equation*}
  \Lambda_{\epsilon, r} = (1- \epsilon^{2}\Delta)^{r/2}.
\end{equation*}
For each $\epsilon > 0$, $\Lambda_{\epsilon,-r} \in \Psi^{-r}$ with
symbol $\langle \epsilon \zeta\rangle^{-r}$; regarded as a subset of
$S^{0}$, this family of symbols is uniformly bounded in $\epsilon$.
In particular, $\Lambda_{\epsilon,-r}\in \Psi^{-r}$ is a uniformly bounded family
in $\Psi^{0}$ converging to the identity map in $\Psi^{s}$ for any
$s>0$.  Similarly, $\Lambda_{\epsilon, r}\in \Psi^{r}$ is a uniformly
bounded family in $\Psi^{r}$.

One useful application of these operators forms the backbone of the
regularization argument.  The monotone
convergence theorem implies that if $u \in H^{s-r}$ and
$\norm{\Lambda_{\epsilon,-r}u}_{H^{s}}$ is uniformly bounded in
$\epsilon$, then in fact $u \in H^{s}$.


Now let $u$ be as in the statement of the Theorem, let $r = N+k + 1$, and apply \eqref{e:propsing} with $u$ replaced by $\Lambda_{\epsilon, -r} u$, $B$ replaced by $\Lambda_{\epsilon, -r} B \Lambda_{\epsilon, r}$, $E$ replaced by $\Lambda_{\epsilon, -r} E \Lambda_{\epsilon, r}$, and $G$ replaced by $\Lambda_{\epsilon, -r} G \Lambda_{\epsilon, r}$; note that $\Lambda_{\epsilon, r}$ commutes with $P$. We are using here the fact that the symbol of $\Lambda_{\epsilon, -r} B \Lambda_{\epsilon, r}$ is controlled by  the symbol of $\Lambda_{\epsilon, -r} E \Lambda_{\epsilon, r}$ through the symbol of $\Lambda_{\epsilon, -r} G \Lambda_{\epsilon, r}$. Note that the corresponding symbol bounds are uniform in $\epsilon$, because for every $r$ and $\beta$ we have
\[
 |\partial_\zeta^\beta \langle \epsilon \zeta \rangle^r| \le C_{r,\beta}  \langle \epsilon \zeta \rangle^{r-|\beta|},
\]
uniformly in $\epsilon$.
Hence  an inspection of the proof of Lemma \ref{main-lemma} and of the inductive
argument shows that we obtain
\[
 \|\Lambda_{\epsilon, -r} B u\|^2_{L^2} \le C(\|\Lambda_{\epsilon, -r} E u\|^2_{L^2} + \|\Lambda_{\epsilon, -r} GPu\|^2_{L^2} + \|\Lambda_{\epsilon, -r} u\|^2_{H^{-N}}),
\]
with a constant $C$ independent of $\epsilon$. Taking $\epsilon \to 0$, and using the monotone convergence theorem argument mentioned in the last paragraph, finishes the proof of the Theorem.

\noindent\textbf{Acknowledgments.} The authors are  grateful to the
participants of the MATRIX workshop \textit{Hyperbolic Differential
  Equations in Geometry and Physics} for their interest in this topic
and for several lively discussions which were the basis of this
note. Thanks also to Semyon Dyatlov for helpful comments and corrections. DB was supported by NSF CAREER grant
DMS-1654056.  Part of this work was supported by NSF grant DMS-1928930 while DB was in
residence at the Mathematical Sciences Research Institute during Summer
2022.  KD was supported by NSF grant DMS-1708511.

\end{document}